\documentclass[a4paper,11pt]{amsart}
\usepackage[margin=1in]{geometry}
\usepackage{amsmath,amssymb,mathtools,microtype}
\usepackage[hidelinks]{hyperref}

\newtheorem{theorem}{Theorem}[section]
\newtheorem{lemma}[theorem]{Lemma}

\newtheorem{proposition}[theorem]{Proposition}
\theoremstyle{remark}
\newtheorem{remark}[theorem]{Remark}

\newcommand{\C}{\mathbb C}
\newcommand{\D}{\mathbb D}
\newcommand{\eps}{\varepsilon}
\newcommand{\e}{\mathrm e}

\title{Exterior Power Sums}
\author{
	Yanping Luo (Sichuan University) \\
	\texttt{2024222010045@stu.scu.edu.cn} \\[6pt]
	Ruiyi Yang (Fudan University) \\
	\texttt{23307090018@m.fudan.edu.cn} \\[6pt]
	Keheng Zhu (Capital Normal University) \\
	\texttt{2240502168@cnu.edu.cn}
}
\date{}

\begin{document}
\begin{abstract}
  We prove that for every fixed $\lambda>0$ and all sufficiently large $n$, any $z_1,\dots,z_n\in\C$ with $|z_j|\geq1$ satisfy
$\max_{2\leq k\leq n+1}|\sum_j z_j^k|>e^{-\lambda n}$.  Consequently, the $n$th root of the optimal maximum tends to $1$, so no constant $C>1$ in Erd\H{o}s 973 can exist.  The proof combines a truncated exponential factorization with overconvergence on an open set outside the unit disk and a normal-family obstruction for Cauchy transforms.
\end{abstract}
\maketitle

\section{Introduction}

For $z_1,\dots,z_n\in\C$, write
\[
 p_k(z)=\sum_{j=1}^n z_j^k,
 \qquad
 P_n(z)=\max_{2\leq k\leq n+1}|p_k(z)|.
\]
Erd\H{o}s asked whether there is a constant $C>1$ such that, for every $n\geq2$, one can choose $z_1=1$ and $|z_j|\geq1$ with $P_n(z)<C^{-n}$; see Erd\H{o}s~\cite[p.~213]{Erdos65} and Hayman~\cite[Problem~7.3]{Hayman74}.  For the corresponding problem with the points in the closed unit disk, Erd\H{o}s obtained an exponentially small construction; Tur\'an records this on p.~35 of~\cite{Turan84}, and L.~Erd\H{o}s later determined a much sharper exponential interval~\cite{Erdos92}.  In the exterior problem, Tur\'an's first main theorem gives the lower bound
\[
 P_n(z)\geq (2e)^{-(1+o(1))n};
\]
see~\cite[Theorem~6.1]{Turan84} and the discussion in~\cite{Bloom973}.  The problem was still recorded as open at the time of writing~\cite{Bloom973}. We use two minima, one without and one with the normalization $z_1=1$:
\[
 \Delta_n=\inf_{\substack{z_1,\dots,z_n\in\C\\ |z_j|\geq1}}P_n(z),
 \qquad
 \Delta_n^*=\inf_{\substack{z_1=1,\ z_2,\dots,z_n\in\C\\ |z_j|\geq1}}P_n(z).
\]
Our result is the following.

\begin{theorem}\label{thm:main}
For every $\lambda>0$ there exists $N(\lambda)$ such that, whenever $n\geq N(\lambda)$ and $z_1,\dots,z_n\in\C$ satisfy $|z_j|\geq1$, one has
\[
 \max_{2\leq k\leq n+1}\left|\sum_{j=1}^n z_j^k\right|>e^{-\lambda n}.
\]
Consequently,
\[
 \lim_{n\to\infty}\Delta_n^{1/n}
 =\lim_{n\to\infty}(\Delta_n^*)^{1/n}=1.
\]
In particular, no constant $C>1$ has the property asked for by Erd\H{o}s.
\end{theorem}

Turturean~\cite{Turturean26} recently proved
\[
 \Delta_n^*\geq
 \exp\!\left(-(0.0597505151\ldots+o(1))n\right),
\]
which implies that any affirmative constant would have to satisfy $C\leq1.0615716674\ldots$.  Several ingredients used below already occur in that manuscript: the representation of
$F(t)=\prod_j(1-z_jt)$ as an exponential times a coefficient-small perturbation, the resulting coefficient identity at degree $n+1$, the estimate $|p_1|\leq(1+o(1))n$, and the inequality $\lambda\leq r-1-\log r$ for a subsequential limit $r=|p_1|/n$; see~\cite[Sections~2--3]{Turturean26}.  We reproduce the precise statements needed here with proofs.

The leading coefficient of $F$ shows that the normalized first power sum cannot tend to zero; this also removes the condition $z_1=1$.  The exponential approximation is made not merely at boundary points but uniformly on a fixed open disk lying outside $\overline\D$.  On that disk the normalized logarithmic derivatives converge to a nonzero constant.  Writing $\alpha_j=z_j^{-1}\in\overline\D$, those logarithmic derivatives become
\[
 \frac1n\frac{F'(w)}{F(w)}
 =\frac1n\sum_{j=1}^n\frac1{w-\alpha_j},
\]
Cauchy transforms of probability measures supported in $\overline\D$.  Normality and the identity theorem then force every constant limit to be zero.  This contradiction proves the theorem.

\section{Algebraic preparation}

We argue by contradiction along an unbounded sequence of integers.  Fix $\lambda>0$, and suppose that for each $n$ in this sequence there are $z_{1,n},\dots,z_{n,n}$ with $|z_{j,n}|\geq1$ and
\begin{equation}\label{eq:hypothesis}
 \eps_n:=\max_{2\leq k\leq n+1}|p_{k,n}|\leq\e^{-\lambda n},
 \qquad
 p_{k,n}:=\sum_{j=1}^n z_{j,n}^k.
\end{equation}
We suppress the second subscript when no confusion can arise.  Put
\[
 F_n(t)=\prod_{j=1}^n(1-z_{j,n}t),
 \qquad x_n=-p_{1,n},
\]
and define
\[
 D_n(t)=\sum_{k=2}^{n+1}\frac{p_{k,n}}{k}t^k,
 \qquad
 A_n(t)=e^{-D_n(t)}=\sum_{\ell\geq0}d_{\ell,n}t^\ell.
\]
For $R\geq0$, let
\[
 b_n(R)=\sum_{k=2}^{n+1}\frac{R^k}{k}.
\]

\begin{lemma}\label{lem:factor}
As formal power series through degree $n+1$:
\[
 F_n(t)\equiv e^{x_nt}A_n(t)\pmod {t^{n+2}}.
\]
Consequently, for $0\leq m\leq n$,
\begin{equation}\label{eq:coeff}
 [t^m]F_n(t)=\sum_{\ell=0}^m d_{\ell,n}\frac{x_n^{m-\ell}}{(m-\ell)!},
\end{equation}
and
\begin{equation}\label{eq:endpoint}
 0=\frac{x_n^{n+1}}{(n+1)!}
   +\sum_{\ell=2}^{n+1}d_{\ell,n}
       \frac{x_n^{n+1-\ell}}{(n+1-\ell)!}.
\end{equation}
Moreover, $d_{0,n}=1$, $d_{1,n}=0$, and for every $R\geq0$,
\begin{equation}\label{eq:majorant}
 \sum_{\ell\geq2}|d_{\ell,n}|R^\ell
 \leq \exp\!\bigl(\eps_n b_n(R)\bigr)-1.
\end{equation}
\end{lemma}

\begin{proof}
In the ring $\C[[t]]$,
\[
 \log F_n(t)=\sum_{j=1}^n\log(1-z_{j,n}t)
 =-\sum_{k\geq1}\frac{p_{k,n}}{k}t^k.
\]
Thus $\log F_n(t)=x_nt-D_n(t)+O(t^{n+2})$, and exponentiation gives the stated congruence.  Formula~\eqref{eq:coeff} follows by comparing coefficients.  Since $F_n$ has degree $n$, its coefficient of $t^{n+1}$ is zero; comparison at that degree gives~\eqref{eq:endpoint}, using $d_{0,n}=1$ and $d_{1,n}=0$.

For the majorant, expand the exponential defining $A_n$.  The weighted coefficient norm $\|\sum a_m t^m\|_R=\sum_m|a_m|R^m$ is submultiplicative, and hence
\[
 \sum_{\ell\geq0}|d_{\ell,n}|R^\ell
 \leq \exp\left(\sum_{k=2}^{n+1}\frac{|p_{k,n}|}{k}R^k\right)
 \leq \exp\!\bigl(\eps_n b_n(R)\bigr).
\]
Subtracting the constant coefficient $d_{0,n}=1$ proves~\eqref{eq:majorant}.
\end{proof}

The next lemma identifies the only possible linear scale of the uncontrolled first power sum.

\begin{lemma}\label{lem:scale}
One has
\begin{equation}\label{eq:limsup-x}
 \limsup_{n\to\infty}\frac{|x_n|}{n}\leq1.
\end{equation}
After passage to a subsequence,  write
\begin{equation}\label{eq:c-limit}
 \frac{x_n}{n}\longrightarrow c,
 \qquad r=|c|.
\end{equation}
Then
\begin{equation}\label{eq:r-range}
 0<r<1,
 \qquad
 \lambda\leq r-1-\log r.
\end{equation}
\end{lemma}

\begin{proof}
Suppose first that $|x_n|\geq(1+\delta)n$ along a subsequence where $\delta>0$.  Set $q_n=(n+1)/|x_n|$.  Then $q_n\leq q_0<1$ for all large $n$.  Dividing~\eqref{eq:endpoint} by $x_n^{n+1}/(n+1)!$, taking absolute values, and using
\[
 \frac{(n+1)!}{(n+1-\ell)!\,|x_n|^\ell}
 =\prod_{j=0}^{\ell-1}\frac{n+1-j}{|x_n|}
 \leq q_n^\ell,
\]
give
\[
 1\leq\sum_{\ell=2}^{n+1}|d_{\ell,n}|q_n^\ell
 \leq \exp\!\bigl(\eps_n b_n(q_n)\bigr)-1.
\]
The quantities $b_n(q_n)$ are uniformly bounded because $q_n\leq q_0<1$, whereas $\eps_n\to0$.  The right-hand side therefore tends to zero, a contradiction.  This proves the asserted upper bound.

We may now pass to a subsequence for which $x_n/n\to c$ with $|c|\leq1$.  We show that $c\neq0$.  The leading coefficient of $F_n$ has modulus
\[
 |[t^n]F_n(t)|=\prod_{j=1}^n|z_{j,n}|\geq1.
\]
Using~\eqref{eq:coeff} with $m=n$, then~\eqref{eq:majorant} at $R=1$, gives

\begin{equation}\label{eq:leading-bound}
 1\leq \frac{|x_n|^n}{n!}
 +\e^{|x_n|}\sum_{\ell=2}^n|d_{\ell,n}|
 \leq \frac{|x_n|^n}{n!}
 +\e^{|x_n|}\bigl(\exp(\eps_n b_n(1))-1\bigr).
\end{equation}
If $r=0$, then $|x_n|/n\to0$, and
\[
 \frac{|x_n|^n}{n!}
 \leq\left(\frac{\e|x_n|}{n}\right)^n\longrightarrow0.
\]
Also $b_n(1)\leq1+\log(n+1)$, so $\eps_n b_n(1)\to0$ and
\[
 \e^{|x_n|}\bigl(\exp(\eps_n b_n(1))-1\bigr)
 \leq2\e^{|x_n|}\eps_n b_n(1)
 =\exp(-\lambda n+o(n))\longrightarrow0.
\]
This contradicts~\eqref{eq:leading-bound}; hence $r>0$.  Notice that this is the only place where the assumption $|z_{j,n}|\geq1$ is used before the final Cauchy-transform argument, and it does not require any distinguished point $z_{1,n}=1$.

Finally,~\eqref{eq:endpoint} and~\eqref{eq:majorant} at $R=1$ imply
\begin{equation}\label{eq:endpoint-bound}
 \frac{|x_n|^{n+1}}{(n+1)!}
 \leq\e^{|x_n|}\bigl(\exp(\eps_n b_n(1))-1\bigr)
 \leq2\e^{|x_n|}\eps_n b_n(1)
\end{equation}
for all large $n$.  Since $|x_n|=rn+o(n)$ and $r>0$, Stirling's formula gives
\[
 \log\frac{|x_n|^{n+1}}{(n+1)!}
 =n(1+\log r)+o(n).
\]
The logarithm of the final expression in~\eqref{eq:endpoint-bound} is at most
\[
 |x_n|-\lambda n+\log\bigl(2b_n(1)\bigr)
 =n(r-\lambda)+o(n).
\]
Therefore $1+\log r\leq r-\lambda$, or equivalently
$\lambda\leq r-1-\log r$.  This endpoint estimate is the one proved in~\cite[Section~2]{Turturean26}.  Since $r\leq1$ and $\lambda>0$, the case $r=1$ is impossible, and~\eqref{eq:r-range} follows.
\end{proof}

\section{Exterior overconvergence and the contradiction}

For $m\geq0$, let
\[
 E_m(u)=\sum_{j=0}^m\frac{u^j}{j!},
\]
and set $E_{-1}=0$.  Summing the coefficient identity~\eqref{eq:coeff} gives the exact formula
\begin{equation}\label{eq:evaluation}
 F_n(w)=\sum_{\ell=0}^n d_{\ell,n}w^\ell E_{n-\ell}(x_nw).
\end{equation}

\begin{lemma}\label{lem:exponential-tail}Let $0<Q<1$ and $\eta>0$.  For all sufficiently large $n$ and every $u\in\C$ with $|u|\leq Qn$,
\begin{equation}\label{eq:tail-estimate}
 |\e^u-E_n(u)|+|\e^u-E_{n-1}(u)|
 \leq \exp\!\bigl((1+\log Q+\eta)n\bigr).
\end{equation}
\end{lemma}

\begin{proof}
Choose $Q'$ with $Q<Q'<1$.  For $j\geq n$ and all sufficiently large $n$,
\[
 \frac{|u|^{j+1}/(j+1)!}{|u|^j/j!}
 =\frac{|u|}{j+1}\leq Q'.
\]
Hence each of the two tails is bounded by a fixed geometric factor times its first term:
\[
 |\e^u-E_{n-1}(u)|
 \leq\frac1{1-Q'}\frac{|u|^n}{n!},
 \qquad
 |\e^u-E_n(u)|
 \leq\frac1{1-Q'}\frac{|u|^{n+1}}{(n+1)!}.
\]
Stirling's formula gives
\[
 \log\frac{(Qn)^n}{n!}=n(1+\log Q)+O(\log n),
\]
and the same exponential rate for the second term.  The $O(\log n)$ term and the fixed geometric factor are absorbed by $\eta n$ for large $n$, proving~\eqref{eq:tail-estimate}.
\end{proof}

\begin{lemma}\label{lem:disk-choice}
Assume~\eqref{eq:c-limit} and~\eqref{eq:r-range}.  There exist a closed disk
$K\Subset\{w:|w|>1\}$ with nonempty interior and constants $a,Q,R,\eta>0$ such that
\begin{equation}\label{eq:disk-data}
 R=\max_{w\in K}|w|<\e^\lambda,
 \qquad Q<1,
\end{equation}
\begin{equation}\label{eq:disk-bounds}
 \operatorname{Re}(cw)>a,
 \qquad |cw|<Q
 \quad (w\in K),
\end{equation}
and
\begin{equation}\label{eq:disk-margins}
 a>1+\log Q+5\eta,
 \qquad
 a>Q-\lambda+\log R+5\eta.
\end{equation}
\end{lemma}

\begin{proof}
Choose
\[
 1<\rho<\min\{\e^\lambda,r^{-1}\},
 \qquad
 w_0=\rho\frac{\overline c}{r},
 \qquad q=r\rho.
\]
Then $|w_0|=\rho>1$, $0<q<1$, and $cw_0=q$.  The two gaps
\[
 q-(1+\log q)>0,
 \qquad
 q-(q-\lambda+\log\rho)=\lambda-\log\rho>0
\]
are strict.  Choose $\eta>0$ smaller than one tenth of their minimum.  As a closed disk centered at $w_0$ shrinks to $\{w_0\}$, the minimum of $\operatorname{Re}(cw)$ and the maxima of $|cw|$ and $|w|$ tend respectively to $q,q,$ and $\rho$.  Hence a sufficiently small disk $K$ lies in $\{|w|>1\}$ and admits numbers $a,Q$ for which~\eqref{eq:disk-data}--\eqref{eq:disk-margins} hold.
\end{proof}

\begin{proposition}\label{prop:exterior-approximation}
Assume~\eqref{eq:c-limit} and~\eqref{eq:r-range}.  For the disk $K$ supplied by Lemma~\ref{lem:disk-choice}, there is $\sigma>0$ such that, uniformly for $w\in K$,
\begin{align}
 F_n(w)&=\e^{x_nw}\bigl(1+O(\e^{-\sigma n})\bigr)
 \label{eq:F-approximation}\\
 F_n'(w)&=x_n\e^{x_nw}
 +O\!\left(n\e^{\operatorname{Re}(x_nw)-\sigma n}\right)
 \label{eq:Fprime-approximation}
\end{align}
In particular, $F_n$ has no zero on $K$ for all sufficiently large $n$, and
\begin{equation}\label{eq:logder-convergence}
 \frac{F_n'(w)}{nF_n(w)}\longrightarrow c
\end{equation}
uniformly on $K$.
\end{proposition}

\begin{proof}
Fix the constants in Lemma~\ref{lem:disk-choice}.  By the strict margins in
\eqref{eq:disk-margins}, one may choose $\delta>0$ so small that, with
$Q_1=Q+\delta$,
\begin{equation}\label{eq:adjusted-margins}
 Q_1<1,\qquad
 a-\delta>1+\log Q_1+4\eta,
 \qquad
 a-\delta>Q_1-\lambda+\log R+4\eta.
\end{equation}
Since $\sup_{w\in K}|(x_n/n-c)w|\to0$, for all sufficiently large $n$ and $w\in K$,
\begin{equation}\label{eq:xnw-bounds}
 \operatorname{Re}(x_nw)\geq(a-\delta)n,
 \qquad
 |x_nw|\leq Q_1n.
\end{equation}
Lemma~\ref{lem:exponential-tail}, applied with $Q_1$, now gives
\begin{equation}\label{eq:tail-relative}
 |\e^{x_nw}-E_n(x_nw)|+|\e^{x_nw}-E_{n-1}(x_nw)|
 \leq \e^{\operatorname{Re}(x_nw)-3\eta n}
\end{equation}
uniformly on $K$.

We next estimate the perturbation coefficients at the radius $R>1$.  Since
\[
 b_n(R)\leq\sum_{k=2}^{n+1}R^k\leq\frac{R^{n+2}}{R-1},
\]
we have $\eps_nb_n(R)\to0$ by $R<\e^\lambda$.  Therefore~\eqref{eq:majorant} and the inequality $\e^y-1\leq2y$ for small $y\geq0$ imply
\begin{equation}\label{eq:S-bound-revised}
 S_n(R):=\sum_{\ell\geq2}|d_{\ell,n}|R^\ell
 \leq\exp\!\bigl((-\lambda+\log R+\eta)n\bigr)
\end{equation}
for all sufficiently large $n$.

Write~\eqref{eq:evaluation} as
\[
 F_n(w)=E_n(x_nw)+G_n(w),
 \qquad
 G_n(w)=\sum_{\ell=2}^n
 d_{\ell,n}w^\ell E_{n-\ell}(x_nw).
\]
The elementary bound $|E_m(u)|\leq\e^{|u|}$, together with
\eqref{eq:xnw-bounds}, \eqref{eq:S-bound-revised}, and
\eqref{eq:adjusted-margins}, gives
\begin{align*}
 |G_n(w)|
 &\leq \e^{Q_1n}S_n(R)\\
 &\leq\exp\!\bigl((Q_1-\lambda+\log R+\eta)n\bigr)
 \leq\e^{\operatorname{Re}(x_nw)-3\eta n}.
\end{align*}
Together with~\eqref{eq:tail-relative}, this proves~\eqref{eq:F-approximation}.

For the derivative, let $m_K=\min_{w\in K}|w|>1$.  Differentiating the finite sum defining $G_n$, using $|x_n|=O(n)$ from~\eqref{eq:limsup-x}, and estimating $|w|^{\ell-1}\leq R^\ell/m_K$, we obtain
\begin{align*}
 |G_n'(w)|
 &\leq\sum_{\ell=2}^n|d_{\ell,n}|
 \left(\ell|w|^{\ell-1}\e^{|x_nw|}
 +|x_n||w|^\ell\e^{|x_nw|}\right)\\
 &\leq C_K n\e^{Q_1n}S_n(R)
 \leq C_K n\e^{\operatorname{Re}(x_nw)-2\eta n}.
\end{align*}
Since $(E_n(x_nw))'=x_nE_{n-1}(x_nw)$, the second estimate in~\eqref{eq:tail-relative} yields~\eqref{eq:Fprime-approximation}, after absorbing fixed constants and taking, for instance, $\sigma=\eta$.  Dividing~\eqref{eq:Fprime-approximation} by $n$ times~\eqref{eq:F-approximation} gives~\eqref{eq:logder-convergence}, because $x_n/n\to c$.
\end{proof}

\section{Cauchy transforms and the conclusion}

We use Montel's theorem and the identity theorem in their usual forms; see~\cite[Chapter~V]{Conway78}.

\begin{lemma}\label{lem:cauchy-obstruction}
Let $\alpha_{1,n},\dots,\alpha_{n,n}\in\overline\D$ and define
\[
 C_n(w)=\frac1n\sum_{j=1}^n\frac1{w-\alpha_{j,n}}
 \qquad (w\in\Omega:=\C\setminus\overline\D).
\]
Suppose that $C_n$ converges locally uniformly to a constant $c$ on a nonempty open subset of $\Omega$.  Then $c=0$.
\end{lemma}

\begin{proof}
The functions $C_n$ are the Cauchy transforms of the probability measures
$n^{-1}\sum_j\delta_{\alpha_{j,n}}$.  For $w\in\Omega$,
\begin{equation}\label{eq:cauchy-bound}
 |C_n(w)|\leq\frac1n\sum_{j=1}^n
 \frac1{|w|-|\alpha_{j,n}|}
 \leq\frac1{|w|-1}.
\end{equation}
Thus $(C_n)$ is locally bounded on $\Omega$ and hence is a normal family.  Choose a subsequence converging locally uniformly on all of $\Omega$ to a holomorphic function $C$.  On the given open subset this subsequence also converges to $c$, so the identity theorem gives $C\equiv c$ on the connected domain $\Omega$.  Passing to the limit in~\eqref{eq:cauchy-bound} yields
$|c|\leq(|w|-1)^{-1}$ for every $|w|>1$.  Letting $|w|\to\infty$ gives $c=0$.
\end{proof}

\begin{proof}[Proof of Theorem~\ref{thm:main}]
Under the contradictory hypothesis~\eqref{eq:hypothesis}, Lemma~\ref{lem:scale} supplies a subsequence for which $x_n/n\to c$ with $0<|c|<1$.  Set
\[
 \alpha_{j,n}=z_{j,n}^{-1}\in\overline\D.
\]
Because
\[
 F_n(w)=(-1)^n\left(\prod_{j=1}^n z_{j,n}\right)
 \prod_{j=1}^n(w-\alpha_{j,n}),
\]
we have on $\Omega=\C\setminus\overline\D$
\[
 \frac1n\frac{F_n'(w)}{F_n(w)}
 =\frac1n\sum_{j=1}^n\frac1{w-\alpha_{j,n}}.
\]
Proposition~\ref{prop:exterior-approximation} says that the left side converges uniformly to the nonzero constant $c$ on the interior of $K$.  Lemma~\ref{lem:cauchy-obstruction} says that such a constant must be zero, a contradiction.  Hence for every fixed $\lambda>0$ only finitely many $n$ admit a configuration satisfying~\eqref{eq:hypothesis}.  This proves the first assertion.

It remains to identify the exponential scale.  The first assertion gives, for every $\lambda>0$ and all sufficiently large $n$,
\[
 \Delta_n\geq\e^{-\lambda n},
 \qquad
 \Delta_n^*\geq\e^{-\lambda n}.
\]
Taking $z_1=\cdots=z_n=1$ gives
$\Delta_n\leq\Delta_n^*\leq n$.  Therefore
\[
 -\lambda\leq\liminf_{n\to\infty}\frac1n\log\Delta_n
 \leq\limsup_{n\to\infty}\frac1n\log\Delta_n\leq0,
\]
and the same inequalities hold for $\Delta_n^*$.  Letting $\lambda\downarrow0$ proves both $n$th-root limits.  Finally, if a fixed $C>1$ satisfied the requirement in Erd\H{o}s's question, then taking $\lambda=\log C$ would contradict the first assertion for every sufficiently large $n$.
\end{proof}

\begin{remark}
The theorem determines the exponential scale but not the finer order of $\Delta_n$ or $\Delta_n^*$.  It remains open within the present argument whether these quantities are bounded away from zero or decay on a subexponential, for example polynomial, scale.
\end{remark}

\section*{Disclosure of automated assistance}
GPT-5.6 Sol was used during proof exploration, source comparison, and drafting.  The written argument above is self-contained and does not use numerical computation or the output of a formal prover as a mathematical premise.  A separate Lean~4 development is available at
\url{https://github.com/hexistartop-gif/Erdos-973_solution_check-by-Lean}; any statement about the extent of formal verification should be read according to the declarations, assumptions, and remaining axioms recorded in that repository.  Neither automated assistance nor formalization is a substitute for independent mathematical review.

\end{document}